\documentclass[a4paper, reqno]{amsart}
\usepackage{amssymb, amsmath, amsthm,enumerate, mathrsfs, mathtools, tikz}
\usetikzlibrary{arrows,calc}

\usepackage[framemethod=TikZ]{mdframed}

\newtheorem{theorem}{Theorem}[section]

\newtheorem{proposition}{Proposition}[section]

\numberwithin{equation}{section}

\theoremstyle{definition}

\newtheorem*{remark}{Remark}

\newcommand{\R}{\mathbb{R}}

\renewcommand{\S}{\mathbb{S}}
\newcommand{\ud}{\mathrm{d}}

\newcommand{\eps}{\varepsilon}
\newcommand{\A}{\mathscr{A}_{\eps(\cdot)}}
\newcommand{\M}{\mathscr{M}}

\title[Maximal Fourier restriction for spheres]{A note on maximal Fourier Restriction for spheres in all dimensions}
\author{ Marco Vitturi }
\address{Marco Vitturi, Laboratoire de  Math\'{e}matiques Jean Leray, 2, rue de la Houssini\`ere BP 92208
44322 Nantes Cedex 3.}
\email{marco.vitturi@univ-nantes.fr}

\date{March 28, 2017}

\begin{document}

\begin{abstract}
We prove a maximal Fourier restriction theorem for the sphere $\mathbb{S}^{d-1}$ in $\mathbb{R}^{d}$ for any dimension $d\geq 3$ in a restricted range of exponents given by the Stein-Tomas theorem. The proof consists of a simple observation. When $d=3$ the range corresponds exactly to the full Stein-Tomas one, but is otherwise a proper subset when $d>3$. We also present an application regarding the Lebesgue points of functions in $\mathcal{F}(L^p)$ when $p$ is sufficiently close to 1.
\end{abstract}

\maketitle

\section{introduction}
Very recently a new line of investigation in the field of Fourier restriction studies has been opened by M\"{u}ller, Ricci and Wright in \cite{MullerRicciWright}, namely that of maximal Fourier restriction theorems. The goal of such line of investigation is to study the Lebesgue points of the Fourier transform $\widehat{f}$ of a generic function $f \in L^p$ when $1\leq p \leq 2$ and sufficiently close to $1$. In the aforementioned paper they prove that for the case of curves in $\R^2$ the following holds:
\begin{theorem}[\cite{MullerRicciWright}]\label{MullerRicciWright_thm1}
Let $\Gamma$ be a $C^2$ curve in $\R^2$ and let $f \in L^p$ with $1\leq p < 8/7$. Then, with respect to arclength measure, a.e. point of $\Gamma$ where the curvature does not vanish is a Lebesgue point of $\widehat{f}$.
\end{theorem}
In particular, if $\mathcal{R}$ is the Fourier restriction operator associated to $\Gamma$, one has from the above that for all $f \in L^p$ with $1\leq p <8/7$
\[ \mathcal{R}f = \widehat{f}\big|_{\Gamma} \]
for a.e. point of $\Gamma$ where the curvature is non zero (a.e. with respect to arclength measure). Theorem \ref{MullerRicciWright_thm1} is the consequence of a clever trick (which we have included in the proof of Proposition \ref{Lebesgue_Points_sphere}, for the reader's convenience) and the following
\begin{theorem}[\cite{MullerRicciWright}]\label{MullerRicciWright_thm2}
Let $\Gamma$ be the graph of a $C^2$ function $\gamma \,:\, I \rightarrow \R$, where $I$ is a bounded interval, and let $\mu$ denote the Affine measure on $\Gamma$, which for $\Gamma$ in this form is given by
\[ \ud\mu ((\xi,\gamma(\xi))) = |\gamma''(\xi)|^{1/3} \ud \xi. \]
Let $\chi \in \mathscr{S}(\R^2)$ be a fixed Schwartz function with $\int \chi = 1$ and define the maximal Fourier restriction operator 
\[ \mathcal{M} f(\xi) := \sup_{\varepsilon,\delta > 0} \Big|\int \int_{\R^2} \widehat{f}(\xi + s, \gamma(\xi) + t) \chi_{\eps}(s) \chi_{\delta}(t) \ud s \ud t\Big|. \]
Then the estimate 
\begin{equation}\label{max_estimate_d=2}
 \|\mathcal{M} f\|_{L^q(\Gamma, \ud \mu)} \lesssim_{p,q} \|f\|_{L^p(\R^2)} 
 \end{equation}
holds for all $f\in L^p$ with $1\leq p < 4/3$ and $p' \geq 3q$. 
\end{theorem}
We have used in the statement the standard notation
\[ \chi_\eps (x) = \frac{1}{\eps} \chi\big(\frac{x}{\eps}\big). \]
Observe that the range of exponents for which \eqref{max_estimate_d=2} holds is the same as that for the usual operator of Fourier restriction to $\Gamma$. The proof of the above theorem follows the lines of Sj\"{o}lin's proof of the Fourier restriction conjecture for curves in the plane as given in \cite{Sjolin}.\\
In this short note we consider the case of Fourier restriction to the sphere $\S^{d-1}$ immersed in $d$-dimensional euclidean space, $d\geq 3$. Let $\ud \sigma$ denote the surface measure of the sphere. Define, analogously to the above, the maximal Fourier restriction operator for the sphere
\begin{equation*}
\mathscr{M} f(\omega) := \sup_{\eps>0} \Big|\int \widehat{f}(\omega + y) \chi_\eps (y) \ud y\Big|, 
\end{equation*}
where $\omega$ ranges over $\S^{d-1}$ and $\chi \in \mathscr{S}(\R^d)$ is a fixed Schwartz function with $\int \chi =1$, $\chi_\eps(y) = \eps^{-d} \chi(y / \eps)$. Then we have the following
\begin{theorem}\label{main_theorem}
Let $d\geq 3$. The operator $\M$ satisfies 
\begin{equation}\label{max_estimate_sphere} \| \M f\|_{L^q(\S^{d-1}, \ud \sigma)} \lesssim_{p,q,d} \|f\|_{L^p(\R^d)}
\end{equation}
for $1\leq p \leq 4/3$, $p' \geq \frac{d+1}{d-1}q$.\\
Moreover, for $1\leq p \leq 8/7$, if $f\in L^p(\R^d)$ then $\sigma$-a.e. point of $\S^{d-1}$ is a Lebesgue point for $\widehat{f}$.
\end{theorem}
Observe that when $d=3$ the range of exponents for which \eqref{max_estimate_sphere} holds has endpoint $L^{4/3}\rightarrow L^2$, and thus corresponds exactly to the Stein-Tomas range for this dimension. For larger values of $d$, the stated range is however only a subset of the full Stein-Tomas range, which is 
\[ 1\leq p \leq \frac{2(n+1)}{n+3}, \quad p' \geq \frac{d+1}{d-1}q.\]
It is precisely the fact that the adjoint estimate to $L^{4/3}\rightarrow L^{q}$ is $L^{q'}\rightarrow L^4$ that allows for a simple proof of the theorem, since $4$ is an even integer and thus the restriction estimate can be restated in cancellation-free form as in \eqref{steintomas_combinatorial} below. Indeed, one can prove the Stein-Tomas theorem in $d=3$ just by the coarea formula (see \cite{DOberlin}).\\
The author is indebted to Diogo Oliveira e Silva for suggesting to look at the special exponents under consideration here.

\section{Proof of the result}
We divide the proof of Theorem \ref{main_theorem} in two by proving separately two propositions. First we prove 
\begin{proposition}\label{max_estimate_proposition}
Let $d\geq 3$. The operator $\M$ satisfies 
\begin{equation}\label{max_estimate_sphere} \| \M f\|_{L^q(\S^{d-1}, \ud \sigma)} \lesssim_{p,q,d} \|f\|_{L^p(\R^d)}
\end{equation}
for $1\leq p \leq 4/3$, $p' \geq \frac{d+1}{d-1}q$.
\end{proposition}
\begin{proof}
It suffices to prove the endpoint, that is $p=4/3$ and $q=4\frac{d+1}{d-1}$. Let 
\[ q_d := 4\frac{d+1}{d-1}. \]
By the Stein-Tomas theorem (see e.g. \cite{Stein}) one has that for the sphere $\mathbb{S}^{d-1}$ immersed in $\mathbb{R}^d$ it holds that for every $f \in L^{4/3}(\R^d)$
\begin{equation}\label{SteinTomas_estimate} 
\|\widehat{f}\|_{L^{q_d}(\S^{d-1}, \ud \sigma)} \lesssim \|f\|_{L^{4/3}(\R^d)}. 
\end{equation}
By duality this is equivalent to the estimate
\[ \|\widehat{g \ud \sigma}\|_{L^4(\R^d)}\lesssim \|g\|_{L^{q'_d}(\S^{d-1},\ud\sigma)}. \]
The numerology here is particularly fortunate since $4$ is an even exponent, which allows us to multilinearise and use Plancherel to write
\[ \|\widehat{g \ud \sigma}\|_{L^4(\R^d)}^2 = \|(\widehat{g \ud \sigma})^2\|_{L^2(\R^d)} = \|g\ud\sigma \ast g\ud\sigma\|_{L^2(\R^d)} \]
(here of course the $L^2$ norm on the right hand side has to be interpreted as the operator norm of the linear operator given by $h \mapsto \langle g\ud\sigma \ast g\ud\sigma, h\rangle$). Thus the Stein-Tomas type estimate \eqref{SteinTomas_estimate} can be stated equivalently in this case as 
\[ \|g\ud\sigma \ast g\ud\sigma\|_{L^2(\R^d)} \lesssim \|g\|_{L^{q'_d}(\S^{d-1},\ud\sigma)}^2, \]
which means
\begin{equation}\label{steintomas_combinatorial}
\Big| \int_{\S^{d-1}} \int_{\S^{d-1}} g(\omega)g(\omega') h(\omega' - \omega) \ud\sigma(\omega) \ud\sigma(\omega') \Big| \lesssim \|g\|_{L^{q'_d}(\S^{d-1},\ud\sigma)}^2 \|h\|_{L^2(\R^d)}.
\end{equation}
We linearise the maximal operator $\M$ by defining 
\[ \A f(\omega) := \int_{\R^d} \widehat{f}(\omega + y) \chi_{\eps(\omega)} (y) \ud y, \]
where $\eps(\cdot)$ an arbitrary measurable function that takes positive values. To bound $\M$ it suffices to bound $\A$ in the same range independently of $\eps(\cdot)$. The desired inequality
\[ \| \A f \|_{L^{q_d}(\S^{d-1},\ud\sigma)}\lesssim \|f\|_{L^{4/3}(\R^d)} \]
is equivalent by duality to the inequality 
\[ \|\A^\ast g\|_{L^4(\R^d)} \lesssim \|g\|_{L^{q'_d}(\S^{d-1},\ud\sigma)}, \]
where $\A^\ast$ is the formal adjoint of $\A$, which is given by 
\[ \A^\ast g(x) := \int_{\S^{d-1}} g(\omega) e^{i \omega\cdot x} \widehat{\chi}(\eps(\omega)x)\ud \sigma(\omega).\]
As before, this is equivalent to establishing 
\[ \|\widehat{\A^\ast g} \ast \widehat{\A^\ast g}\|_{L^2(\R^d)} \lesssim \|g\|_{L^{q'_d}(\S^{d-1},\ud\sigma)}^2. \]
First of all, observe that by Fubini's theorem
\begin{align*}
\widehat{\A^\ast g}(\xi) &= \int e^{-i x\cdot \xi} \int_{\S^{d-1}} g(\omega) e^{i \omega\cdot x} \widehat{\chi}(\eps(\omega)x)\ud \sigma(\omega) \ud x \\
&= \int_{\S^{d-1}} g(\omega) \int e^{-i x \cdot (\xi - \omega)} \widehat{\chi}(\eps(\omega)x) \ud x \ud \sigma(\omega) \\
&= \int_{\S^{d-1}} g(\omega) \chi_{\eps(\omega)}(\xi - \omega) \ud \sigma(\omega)
\end{align*}
(with a little abuse of notation).
Let then $h \in L^2(\R^d)$, so that by the above observation and multiple applications of Fubini's theorem we have the following chain of equalities:
\begin{align*}
\langle \widehat{\A^\ast g} &\ast \widehat{\A^\ast g}, h \rangle = \int \int \widehat{\A^\ast g}(\xi-\eta)\widehat{\A^\ast g}(\eta)\overline{h(\xi)} \ud \eta  \ud \xi \\
&= \int\int \int_{\S^{d-1}}\int_{\S^{d-1}} g(\omega)g(\omega') \chi_{\eps(\omega)}(\xi-\eta-\omega)\chi_{\eps(\omega')}(\eta-\omega') \ud\sigma(\omega) \ud\sigma(\omega') \overline{h(\xi)} \ud \eta  \ud \xi \\
&= \int \int_{\S^{d-1}}\int_{\S^{d-1}} g(\omega)g(\omega')\overline{h(\xi)} \Big(\int \chi_{\eps(\omega)}(\xi-\eta-\omega)\chi_{\eps(\omega')}(\eta-\omega') \ud\eta\Big) \ud\sigma(\omega) \ud\sigma(\omega') \ud \xi \\
&= \int \int_{\S^{d-1}}\int_{\S^{d-1}} g(\omega)g(\omega')\overline{h(\xi)} (\chi_{\eps(\omega)}\ast \chi_{\eps(\omega')})(\xi+\omega'-\omega) \ud\sigma(\omega) \ud\sigma(\omega') \ud \xi \\
&= \int_{\S^{d-1}}\int_{\S^{d-1}} g(\omega)g(\omega') \int \tilde{h}(-\xi) (\chi_{\eps(\omega)}\ast \chi_{\eps(\omega')})(\xi+\omega'-\omega)\ud \xi \ud\sigma(\omega) \ud\sigma(\omega')  \\
&= \int_{\S^{d-1}}\int_{\S^{d-1}} g(\omega)g(\omega') (\tilde{h}\ast\chi_{\eps(\omega)}\ast \chi_{\eps(\omega')})(\omega'-\omega) \ud\sigma(\omega) \ud\sigma(\omega'),
\end{align*} 
where $\tilde{h}(\xi) = \overline{h(-\xi)}$. But then we have that pointwise 
\[ |(\tilde{h}\ast\chi_{\eps(\omega)}\ast \chi_{\eps(\omega')})(\omega'-\omega)| \lesssim M^2 \tilde{h}(\omega' - \omega), \]
with constant depending only on the choice of $\chi$, where $M$ is the Hardy-Littlewood maximal function; therefore by the Stein-Tomas restriction estimate \eqref{steintomas_combinatorial} we have
\begin{align*}
|\langle \widehat{\A^\ast g} \ast \widehat{\A^\ast g}, h \rangle| &\lesssim   \int_{\S^{d-1}}\int_{\S^{d-1}} |g(\omega)||g(\omega')| M^2 \tilde{h}(\omega' - \omega) \ud\sigma(\omega) \ud\sigma(\omega') \\
&\lesssim \|g\|_{L^{q'_d}(\S^{d-1},\ud\sigma)}^2 \|M^2 \tilde{h}\|_{L^2(\R^d)}\lesssim \|g\|_{L^{q'_d}(\S^{d-1},\ud\sigma)}^2 \|h\|_{L^2(\R^d)},
\end{align*}
which proves the desired estimate for $\A$.
\end{proof}

\begin{remark}
It is interesting to notice that the critical endpoint for Fourier restriction to curves in $\R^2$ is $L^{4/3} \rightarrow L^{4/3}$, and we know that the corresponding (even restricted) strong type estimate is false by work of  Beckner, Carbery, Semmes and Soria \cite{BecknerCarberySemmesSoria}. Thus the proof above barely misses to adapt to case $d=2$.
\end{remark}

Finally, we prove the second half of Theorem \ref{main_theorem}, restated below.
\begin{proposition}\label{Lebesgue_Points_sphere}
Let $1 \leq p \leq 8/7$. If $f\in L^p(\R^d)$ then $\sigma$-a.e. point of $\S^{d-1}$ is a Lebesgue point for $\widehat{f}$.
\end{proposition}
\begin{proof}
The proof that follows is taken from \cite{MullerRicciWright} and has been included only for the reader's convenience.\\
Let $\mathcal{R}$ denote the Fourier restriction operator to the sphere $\S^{d-1}$. Let $\M^{+}$ denote the positive maximal Fourier restriction operator associated to the sphere $\S^{d-1}$, defined as 
\[ \M^{+} f(\omega) := \sup_{\eps>0} \frac{1}{\eps^d}\int_{|y|\leq \eps} |\widehat{f}(\omega + y)| \ud y. \] 
To prove the proposition it suffices to show that 
\begin{equation}\label{pos_max_estimate}
\| \M^{+} f \|_{L^q (\S^{d-1},\ud\sigma)} \lesssim_{p,q} \|f\|_{L^p(\R^d)}
\end{equation}
for $1\leq p \leq 8/7$ and $p' \geq q \frac{d+1}{d-1}$. Indeed, assuming this holds, one can define 
\[ F(\omega) := \limsup_{\eps\rightarrow 0} \frac{1}{\eps^d} \int_{|y|\leq \eps} |\widehat{f}(\omega + y) - \mathcal{R}f(\omega)| \ud y; \]
since $\mathcal{R}\varphi = \widehat{\varphi}|_{\S^{d-1}}$ for any $\varphi \in \mathscr{S}(\R^d)$, we have 
\begin{align*}
F(\omega)& \leq \limsup_{\eps\rightarrow 0} \frac{1}{\eps^d} \int_{|y|\leq \eps} |\widehat{f - \varphi}(\omega + y)| \ud y + |\mathcal{R}(f - \varphi)(\omega)| \\
& \leq \M^{+}(f-\varphi)(\omega) + |\mathcal{R}(f - \varphi)(\omega)|.
\end{align*}
By the Stein-Tomas estimate and \eqref{pos_max_estimate} it follows then that 
\[ \|F\|_{L^q(\S^{d-1},\ud\sigma)} \lesssim \|f-\varphi\|_{L^p(\R^d)} \]
in the given range, and by taking $\varphi$ to be an approximant of $f$ in $L^p$ norm we see that $\|F\|_{L^q(\S^{d-1},\ud\sigma)} = 0$ or equivalently that $F = 0$ $\sigma$-a.e., which proves the proposition. Thus it suffices to prove \eqref{pos_max_estimate}, and in particular it suffices to prove it under the assumption that $p' = q \frac{d+1}{d-1}$. This will follow from Proposition \ref{max_estimate_proposition}.\\
Observe that by H\"{o}lder's inequality we have 
\[ \frac{1}{\eps^d}\int_{|y|\leq \eps} |\widehat{f}(\omega + y)| \ud y \lesssim \Big(\frac{1}{\eps^d}\int_{|y|\leq \eps} |\widehat{f}(\omega + y)|^2 \ud y\Big)^{1/2}; \]
let then $h := f \ast \tilde{f}$, so that 
\[ \widehat{h} = |\widehat{f}|^2, \]
and we have
\[ \M^{+} f \lesssim (\M h)^{1/2}\]
pointwise. Let $s$ be such that $s \leq 4/3$ and
\[ \frac{q}{2}\frac{d+1}{d-1} = s'; \]
by Proposition \ref{max_estimate_proposition} we have then 
\begin{align*}
 \|\M^{+} f\|_{L^q(\S^{d-1},\ud\sigma)} &\lesssim \|\M h\|_{L^{q/2}(\S^{d-1},\ud\sigma)}^{1/2} \lesssim \|h\|_{L^s(\R^d)}^{1/2} \\
 & \leq \|f\|_{L^p(\R^d)},  
\end{align*}
where $1 + \frac{1}{s} = \frac{2}{p}$ and the last inequality is an application of Young's inequality. Thus it follows that 
\[ p' = 2s' = q \frac{d+1}{d-1}, \]
as desired. Since $s\leq 4/3$, we see that we can only afford $p \leq 8/7$, and this concludes the proof. 
\end{proof}

\bibliography{partial_stein_tomas_maximal_restriction}
\bibliographystyle{amsplain}

\end{document}